\documentclass[a4paper]{amsart}

\usepackage{amsmath,amscd,color}
\usepackage{amssymb}

\usepackage[top=1.4in, bottom=1.4in, left=1.4in, right=1.4in]{geometry}

\usepackage{amsthm}
\theoremstyle{plain}
\newtheorem{mydef}{Definition}[section]

\newtheorem{mylemma}[mydef]{Lemma}

\newtheorem{myprop}[mydef]{Proposition}

\theoremstyle{definition}

\newtheorem{myremark}[mydef]{Remark}
\theoremstyle{remark}

\usepackage{enumitem}

\usepackage{comment}

\usepackage{cancel}

\usepackage{mathtools}

\usepackage{tikz}
\usetikzlibrary{matrix,arrows}

\title{On the Inner Automorphisms of a Singular Foliation}
\author{Alfonso Garmendia}
\thanks{KU Leuven, Department of Mathematics.}
\thanks{Email: \texttt{alfonso.garmendia@kuleuven.be},  \texttt{ori.yudilevich{@}kuleuven.be}.}
\author{Ori Yudilevich}
\date{November 1, 2018}

\begin{document}

\begin{abstract}
A singular foliation in the sense of Androulidakis and Skandalis is an involutive and locally finitely generated module of compactly supported vector fields on a manifold. An automorphism of a singular foliation is a diffeomorphism that preserves the module. In this note, we give a proof of the (surprisingly non-trivial) fundamental fact that the time-one flow of an element of a singular foliation (i.e. its exponential) is an automorphism of the singular foliation. This fact was previously proven in \cite{Androulidakis2009} using an infinite dimensional argument (involving differential operators), and the purpose of this note is to complement that proof with a finite dimensional proof in which the problem is reduced to solving an elementary ODE.
\end{abstract}

\maketitle

\section{Introduction}

A \textit{singular foliation} on a manifold $M$, in the sense of Androulidakis and Skandalis \cite{Androulidakis2009}, is a $C^\infty(M)$-module of compactly supported vector fields on $M$ that is involutive and locally finitely generated. By the Stefan-Sussmann Theorem, such a module gives rise to a smooth partition of the manifold by immersed submanifolds of possibly varying dimension. However, a given partition may arise from different modules, and the precise choice of a module can be viewed as extra data that encodes the ``dynamics'' along the leaves of the partition. Lie algebra actions on a manifold, and, more generally, Lie algebroids, are an important source of examples of singular foliations in this sense (simply apply the anchor of the Lie algebroid to its module of compactly supported sections).

An automorphism of a singular foliation is a diffeomorphism that pushes forward the module to itself. In \cite{Androulidakis2009} (Proposition 1.6), it was proven that the time one flow of an element of a singular foliation is an automorphism of the foliation -- an \textit{inner automorphism}. This fact is fundamental in the theory. It underlies the construction carried out in \cite{Androulidakis2009} of the so-called \textit{holonomy groupoid} that is canonically associated with any singular foliation, or, more precisely, it is used in the construction of a bi-submersion -- the basic building block of the holonomy groupoid -- out of any given finite set of local generators (see Proposition 2.10 (a) in \cite{Androulidakis2009}). The proof given in \cite{Androulidakis2009} of this fact is of an analytic nature (using the exponential of a certain differential operator), and the purpose of this note is to complement that proof with a ``finite dimensional'' proof in which the problem is reduced to solving an elementary ODE.  

The authors would like to thank Marco Zambon for helpful discussions and suggestions, and to the anonymous referee for his useful recommendations.  The first author was supported by IAP Dygest and by the FWO under EOS project G0H4518N. The second author was supported by the long term structural funding -- Methusalem grant of the Flemish Government. In addition, both authors were supported by the FWO research project G083118N.

\section{Basic Definitions} 

We denote the ring of smooth functions on a manifold $M$ by $C^\infty(M)$ and the $C^\infty(M)$-module of vector fields on $M$ by $\mathfrak{X}(M)$. The support of a function $f\in C^\infty(M)$ is denoted by $\mathrm{supp}(f)$ and, similarly, of a vector field $X\in\mathfrak{X}(M)$ by $\mathrm{supp}(X)\subset M$. We denote the ring of compactly supported smooth functions by $C^\infty_c(M)$. Our main object of interest is the $C^\infty(M)$-module of compactly supported vector fields on $M$ (a submodule of $\mathcal{X}(M)$) and its submodules 
\begin{equation*}
\mathcal{F}\subset \mathfrak{X}_c(M).
\end{equation*}

Let $U\subset M$ be an open subset. Viewing the $C^\infty(U)$-module $\mathfrak{X}_c(U)$ as a subset of $\mathfrak{X}_c(M)$ (by extension by zero), we may define the restriction to $U$ of a submodule $\mathcal{F}\subset \mathfrak{X}_c(M)$ as the submodule
\begin{equation*}
\mathcal{F}_U:= \mathcal{F}\cap \mathfrak{X}_c(U)\subset \mathfrak{X}_c(U).
\end{equation*}
Equivalently, $\mathcal{F}_U = C^\infty_c(U)\mathcal{F}$, where the inclusion $C^\infty_c(U)\mathcal{F} \subset \mathcal{F}\cap \mathfrak{X}_c(U)$ is clear, and the reverse inclusion follows from the fact that for any $X\in \mathfrak{X}_c(U)$ we can find $f\in C^\infty_c(U)$ such that $f|_{\mathrm{supp}(X)}=1$, and hence $X=fX$. 

A submodule $\mathcal{F}\subset \mathfrak{X}_c(M)$ is \textbf{finitely generated} if there exists a finite set of vector fields $Y^1,...,Y^r\in\mathfrak{X}(M)$, called \textit{generators}, such that
\begin{equation*}
\mathcal{F} = \langle Y^1,...,Y^r \rangle_{C^\infty_c(M)},
\end{equation*}
i.e. $\mathcal{F}$ is the $C^\infty_c(M)$-linear span of the generators. Note that the generators of $\mathcal{F}$ are not required to be in $\mathcal{F}$ and their support is not necessarily compact. A submodule $\mathcal{F}\subset \mathfrak{X}_c(M)$ is \textbf{locally finitely generated} if every point $x\in M$ has a neighborhood $U\subset M$ such that $\mathcal{F}_U$ is finitely generated. We will need the following useful lemma:

\begin{mylemma}
\label{lemma:unionoffinite} 
Let $\mathcal{F}\subset \mathfrak{X}_c(M)$ be a submodule.
\begin{enumerate}[wide,labelwidth=!,labelindent=0pt]
	\item[(a)] If $\mathcal{F}$ is finitely generated, then $\mathcal{F}_U$ is finitely generated for all open subsets $U\subset M$.
	\item[(b)] If $U,V\subset M$ are open subsets such that $\mathcal{F}_U$ and $\mathcal{F}_V$ are finitely generated, then $\mathcal{F}_{U\cup V}$ is finitely generated. 
	\item[(c)]  $\mathcal{F}$ is locally finitely generated if and only if, for all $\rho\in C^\infty_c(M)$, the submodule 
	\begin{equation*}
	\rho\mathcal{F}:=\{ \rho X \;|\; X\in\mathcal{F} \}\subset \mathfrak{X}_c(M)
	\end{equation*}
	is finitely generated. 
\end{enumerate}
\end{mylemma}

\begin{proof}
(a) Using that $\mathcal{F}_U=C^\infty_c(U)\mathcal{F}$, one easily sees that that the restrictions of generators of $\mathcal{F}$ to $U$ are generators of $\mathcal{F}_U$. 

(b) Let $X^1,...,X^r\in \mathfrak{X}(U)$ be generators of $\mathcal{F}_U$ and $Y^1,...,Y^s\in \mathfrak{X}(V)$ of $\mathcal{F}_V$, and choose a partition of unity $\{\rho_U,\rho_V\}$ subordinate to the open cover $\{U,V\}$ of $U\cup V$. We show that $\rho_U X^1,...,\rho_U X^r,\rho_V Y^1,...,\rho_V Y^s\in\mathfrak{X}(U\cup V)$ (or rather their extension by zero to $U\cup V$) are generators of $\mathcal{F}_{U\cup V}$, i.e. that $\mathcal{F}_{U\cup V}$ is equal to the $C^\infty_c(U\cup V)$-linear span of the generators. 

On the one hand, if $f\in C^\infty_c(U\cup V)$ then $f\rho_U\in C^\infty_c(U)$, since $\mathrm{supp}(f\rho_U) = \mathrm{supp}(f)\cap\mathrm{supp}(\rho_U)$ and a closed subset of a compact set is compact. Similarly, $f\rho_V\in C^\infty_c(V)$. Thus, using that $\mathcal{F}_U \subset \mathcal{F}_{U\cup V}$ and $\mathcal{F}_V \subset \mathcal{F}_{U\cup V}$, we see that any $C^\infty_c(U\cup V)$-linear combination of the generators is an element of $\mathcal{F}_{U\cup V}$. 

On the other hand, let $Z\in \mathcal{F}_{U\cup V}$ and write $Z = \rho_U Z + \rho_V Z$. For the reason explained above, $\rho_U Z \in \mathfrak{X}_c(U)$ and $\rho_V Z \in \mathfrak{X}_c(V)$. Choose functions $\lambda_U\in C^\infty_c(U)$ and $\lambda_V\in C^\infty_c(V)$ such that $\lambda_U|_{\mathrm{supp}(\rho_U X)}=1$ and $\lambda_V|_{\mathrm{supp}(\rho_V X)}=1$, we have that $Z = \rho_U \lambda_U Z + \rho_V \lambda_V Z$. Finally, $\lambda_U Z = \sum_{i=1}^r f_i X^i$ for some $f^i\in C^\infty_c(U)$ and $\lambda_V Z = \sum_{i=1}^s g_i \textcolor{blue}{Y}^i$ for some $g^i\in C^\infty_c(V)$, and so
\begin{equation*}
Z = \sum_{i=1}^r f_i \, \rho_U X^i + \sum_{i=1}^s g_i \, \rho_U Y^i. 
\end{equation*}

(c) We begin with the forward implication. Let $\rho\in C^\infty_c(M)$. There exists an open subset $U\subset M$ containing $\mathrm{supp}(\rho)$ such that $\mathcal{F}_U$ is finitely generated, namely we may choose a finite covering of $\mathrm{supp}(\rho)$ by open subsets $U_1,...,U_r$ such that all restrictions $\mathcal{F}_{U_i}$ are finitely generated and then apply part (b) a finite number of times to conclude that $\mathcal{F}_{U_1\cup...\cup U_r}$ is finitely generated. It is now easy to see that $\rho \mathcal{F}= \rho \mathcal{F}_U$ and hence $\rho\mathcal{F}$ is finitely generated.

For the reverse implication, given $x\in M$, we may construct an open neighborhood $U$ of $x$ and a $\rho\in C^\infty_c(M)$ such that $\rho|_U =1$, and we immediately see that $\mathcal{F}_U=\rho (\mathcal{F}_U)=(\rho \mathcal{F})_U$, which is finitely generated by part (a).
\end{proof}

\begin{mydef}
	A \textbf{singular foliation} on a manifold $M$ (in the sense of Androulidakis and Skandalis \cite{Androulidakis2009}) is a locally finitely generated submodule $\mathcal{F}\subset \mathfrak{X}_c(M)$ that is involutive, i.e. $[\mathcal{F},\mathcal{F}]\subset \mathcal{F}$. 
\end{mydef}

\section{Inner Automorphisms are Automorphisms}

Let $\mathcal{F}$ be a singular foliation on $M$, and let us denote the 
\textbf{automorphism group} of $\mathcal{F}$ by
\begin{equation*}
\mathrm{Aut}(\mathcal{F}) := \{\; \varphi\in\mathrm{Diff}(M)\;|\; \varphi_*(\mathcal{F})= \mathcal{F} \;\}.
\end{equation*}
Note that $\varphi\in\mathrm{Aut}(\mathcal{F})$ is an automorphism in the sense of $C^\infty(M)$-modules when viewed as the pair $((\varphi^{-1})^*,\varphi_*)$ consisting of the ring morphism $(\varphi^{-1})^*:C^\infty(M)\to C^\infty(M)$ and the module morphism $\varphi_*:\mathcal{F}\to \mathcal{F}$, since $\varphi_*(fX) = (\varphi^{-1})^*(f)\,\varphi_*(X)$ for all $X\in\mathcal{F},\;f\in C^\infty(M)$. 

Due to the compact support, any element $X\in\mathcal{F}$ is a complete vector field (i.e. its flow is defined for all times), and hence, denoting its flow by $\varphi_X^t$, we may define the \textbf{exponential map} $\mathrm{exp}:\mathcal{F}\to \mathrm{Diff}(M)$ by
\begin{equation*}
\mathrm{exp}(X):=\varphi_X^{1}.
\end{equation*} 
We now turn to the main objective of this paper, a proof of the following (surprisingly non-trivial) fact:

\begin{myprop}[Proposition 1.6 of \cite{Androulidakis2009}]
	Let $\mathcal{F}$ be a singular foliation on $M$. Then,
	\begin{equation*}
	\mathrm{exp}(\mathcal{F})\subset \mathrm{Aut}(\mathcal{F}).
	\end{equation*}
\end{myprop}

\begin{proof}
We need to show that $(\varphi_X^1)_*(Y)\in \mathcal{F}$ for all $X,Y\in\mathcal{F}$. Indeed, this implies that $(\varphi_X^1)_*(\mathcal{F})\subset\mathcal{F}$, and since $X\in\mathcal{F} \Rightarrow -X\in\mathcal{F}$ and $(\varphi_X^1)_*((\varphi_{-X}^1)_*(Y))=Y$, also that $(\varphi_X^1)_*(\mathcal{F})=\mathcal{F}$. 

Let then $X,Y\in\mathcal{F}$ and let us shorten the notation for the flow of $X$ to $\varphi^t = \varphi_X^t$. Since $\mathrm{supp}(X)\subset M$ is compact, there exists a precompact open neighborhood $U$ of $\mathrm{supp}(X)$ in $M$. Let $\{\rho_U, \rho_V\}$ be a partition of unity subordinate to the open cover $\{U,V:=M\backslash \mathrm{supp}(X)\}$ of $M$. Since $U$ is precompact, $\rho_U$ has compact support, and hence $\rho_U\mathcal{F}$ is finitely generated by part (c) of Lemma \ref{lemma:unionoffinite}. Fixing generators $Y^1,...,Y^N\in\mathfrak{X}(U)$ of $\rho_U\mathcal{F}$, we may write
\begin{equation*}
Y = \rho_U Y + \rho_V Y =  \sum_{i=1}^N f_i Y^i + \rho_V Y,
\end{equation*}
for some $f_i\in C^\infty_c(U)$. Now, since $\varphi^t\big|_{M\backslash\mathrm{supp}(X)} = \mathrm{Id}$ for all $t$ and hence $(\varphi^1)_*(\rho_VY)=\rho_VY$, we see that the problem is reduced to showing that $(\varphi^1)_*(Y^i) = \sum_j f^i_j Y^j$, for some functions $f^i_j\in C^\infty(U)$. To this end, we compute the following:
\begin{equation*}
\begin{split}
\frac{d}{dt}((\varphi^t)_* Y^i)_x &= \frac{d}{dt}(\varphi^t)_* Y^i_{\varphi^{-t}(x)} \\ &= - \frac{d}{ds}\Big|_{s=0}(\varphi^{t-s})_* Y^i_{\varphi^{-t+s}(x)} \\ &= - (\varphi^t)_* \frac{d}{ds}\Big|_{s=0} (\varphi^{-s})_* Y^i_{\varphi^s(\varphi^{-t}(x))} \\
&= (\varphi^t)_*[Y^i,X]_{\varphi^{-t}(x)}.
\end{split}
\end{equation*}
We claim that $[Y^i,X] \in \mathcal{F}_U$. Indeed, choosing $\lambda\in C^\infty_c(U)$ that is 1 on $\mathrm{supp}(X)$, we have $[Y^i,X]= [Y^i,\lambda X] = Y^i(\lambda) X + \lambda[Y^i,X] = Y^i(\lambda) X + [\lambda Y^i,X]\in \mathcal{F}_U$, where we used that $X(\lambda)=0$, $X\in\mathcal{F}_U$ and $\mathcal{F}$ is involutive. Hence, we can write $[Y^i,X] = \sum_j \gamma^i_j Y^j$ for some functions $\gamma^i_j\in C^\infty_c(M)$. We thus have:
\begin{equation*}
\begin{split}
\frac{d}{dt}(\varphi^t)_* Y^i_{\varphi^{-t}(x)} =  \sum_{j=1}^N \gamma^i_j(\varphi^{-t}(x)) (\varphi^t)_*Y^j_{\varphi^{-t}(x)}.
\end{split}
\end{equation*}
Now, for a fixed $x\in M$, this is a system of $N$ equations indexed by $i$, each an equality  of curves in $T_xM$. Fixing a basis of $T_xM$, every component is a linear first ordinary partial differential equation of the type $\dot{v}(t)=A(t)v(t)$, with $v:I\to \mathbb{R}^N$ and $A:I\to \mathrm{End}(\mathbb{R}^N)$, and its solution is given by $v(t)= e^{\int_0^t A(\epsilon)d\epsilon}v(0)$. Thus, writing $\gamma$ for the $N\times N$ matrix whose entries are $\gamma^i_j$, we have at $t=1$ that:
\begin{equation*}
((\varphi^1)_*Y^i)_x =  (\varphi^1)_*Y^i_{\varphi^{-1}(x)} = \sum_{j=1}^N(e^{\int_0^1\gamma(\varphi^{-\epsilon}(x))d\epsilon})^i_j Y^j_x,
\end{equation*}
where the exponential is the exponential of $N\times N$ matrices. Clearly, the coefficients of $Y^j_x$ in the final expression are smooth functions of $x$, and hence we are done. As a bonus, we have also obtained an explicit formula for $(\varphi^1)_*Y^i$ in terms of the bracket of the $Y^i$'s with $X$ (which is encoded in the $\gamma$ matrices). 
\end{proof}

\begin{myremark}
	The above proof is essentially the proof of the fact that if a smooth (but possibly singular) distribution is involutive and its rank is constant along integral paths of its sections, then it is homogeneous, in the sense that the flow of any of its sections preserves the distribution (see e.g. Theorem 3.5.10 in \cite{Rudolph2013}). To adapt the proof of this fact to our proposition, one has to add the argument that the solution of the ODE one is solving is smooth with respect to the base point (i.e. the initial condition). However, the proof we give above goes a step further by giving an explicit formula in which the smoothness with respect to the base point is evident. 
\end{myremark}

\bibliographystyle{plain}

\def\cprime{$'$} \def\cprime{$'$} \def\cprime{$'$}

\end{document}